\documentclass[11 pt]{amsart}
\usepackage{amssymb}
\usepackage[colorlinks, linkcolor=blue, citecolor=red]{hyperref}
\newtheorem{theorem}{Theorem}[section]

\newtheorem{proposition}[theorem]{Proposition}
\newtheorem{corollary}[theorem]{Corollary}
\theoremstyle{definition}

\newtheorem{example}[theorem]{Example}

\theoremstyle{remark}
\newtheorem*{ack}{Acknowledgement}
\newtheorem{remark}[theorem]{Remark}
\newcommand{\Ker}{\operatorname{Ker}}
\newcommand{\im}{\operatorname{Im}}
\newcommand{\rk}{\operatorname{rk}}

\begin{document}
\title[Group actions on Milnor manifolds ]{Free actions of some compact groups on Milnor manifolds}
\author{PINKA DEY}
\author{MAHENDER SINGH}
\address{Department of Mathematical Sciences, Indian Institute of Science Education and Research (IISER) Mohali, Knowledge City, Sector 81, SAS Nagar, Manauli (PO), Punjab 140306, India.}
\email{pinkadey11@gmail.com, mahender@iisermohali.ac.in.}
\subjclass{Primary 57S25, 57S10; Secondary 55R20, 55M20}
\keywords{Borsuk-Ulam theorem; cohomology algebra; equivariant map; index of involution; Leray-Serre spectral sequence; Milnor manifold; Schwarz genus}

\begin{abstract}
In this paper, we investigate free actions of some compact groups on cohomology real and complex Milnor manifolds. More precisely, we compute the  mod 2 cohomology algebra of the orbit space of an arbitrary free $\mathbb{Z}_2$ and $\mathbb{S}^1$-action on a compact Hausdorff space with mod 2 cohomology algebra of a real or a complex Milnor manifold. As applications, we deduce some Borsuk-Ulam type results for equivariant maps between spheres and these spaces. For the complex case, we obtain a lower bound on the Schwarz genus, which further establishes the existence of coincidence points for maps to the Euclidean plane. 
\end{abstract}
\maketitle

\section{Introduction}
A basic problem in the theory of transformation groups is to determine groups that can act freely on a given topological space. Once we know that a group acts freely on a given space, the next natural problem is to determine all actions of the group up to conjugation. Determining the homeomorphism or homotopy type of the orbit space is, in general, a difficult problem. A non-trivial result of Oliver \cite{oliver} states that the orbit space of any action of a compact Lie group on a Euclidean space is contractible. For spheres, Milnor \cite{milnor} proved that for any free involution on $\mathbb{S}^n$, the orbit space has the homotopy type of $\mathbb{R}P^n$. Free actions of finite groups on spheres, particularly $\mathbb{S}^3$, have been well-studied in the past, see for example  \cite{rice, ritter, rubinstein}. But, not many results are known for compact manifolds other than spheres. In \cite{myers}, Myers investigated orbit spaces of free involutions on three-dimensional lens spaces. In \cite{tao}, Tao determined orbit spaces of free involutions on ${\mathbb{S}^1} \times \mathbb{S}^2$, and Ritter \cite{ritter2} extended these results to free actions of cyclic groups of order $2^n$. Tollefson \cite{tollefson} proved that there are precisely four conjugacy classes of involutions on ${\mathbb{S}^1} \times \mathbb{S}^2$. Fairly recently, Jahren and Kwasik \cite{jahren} classified, up to conjugation, all free involutions on $\mathbb{S}^1 \times \mathbb{S}^n$ for $n \ge 3$, by showing that there are exactly four possible homotopy types of orbit spaces.\par

Various attempts have been made towards the weaker problem of determining possible cohomology algebra of orbit spaces of free actions of finite groups on some specific classes of manifolds, for example, products of spheres, spherical space forms and their products. Dotzel et al. \cite{dotzel} determined the cohomology algebra of orbit spaces of free $\mathbb{Z}_p$ ($p$ prime) and $\mathbb{S}^1$-actions on cohomology product of two spheres. Orbit spaces of free involutions on cohomology lens spaces were investigated by Singh \cite{msinghlens}. The cohomology algebra of orbit spaces of free involutions on product of two projective spaces was computed in another paper \cite{msinghprojective}. Recently, Pergher et al. \cite{pergher} and Mattos et al. \cite{mattos} considered free $\mathbb{Z}_2$ and $\mathbb{S}^1$-actions on spaces of type $(a, b)$, which are certain products or wedge sums of spheres and projective spaces. As applications, they also established some bundle theoretic analogues of Borsuk-Ulam theorem for these spaces. \par

Viewing the product of two projective spaces as a trivial bundle, it is interesting to consider non-trivial projective space bundles over projective spaces. Milnor manifolds are fundamental examples of such spaces.  It is well-known that the unoriented cobordism algebra of smooth manifolds is generated by the cobordism classes of real projective spaces and real Milnor manifolds \cite[Lemma 1]{milnor3}. Therefore, determining various invariants of these manifolds is of interest. Free actions of elementary abelian 2-groups on products of Milnor manifolds have been investigated in \cite{msinghmilnor}, wherein some bounds on the rank of these groups are determined.\par

The aim of this paper is to investigate free $\mathbb{Z}_2$ and $\mathbb{S}^1$-actions on mod 2 cohomology real and complex Milnor manifolds. More precisely, we determine the possible mod 2 cohomology algebra of orbit spaces of free $\mathbb{Z}_2$ and $\mathbb{S}^1$-actions on these spaces. We also find necessary and sufficient conditions for the existence of free actions on these spaces. As applications, we obtain some Borsuk-Ulam kind results for these spaces. We also determine some lower bound on the genus and results on the existence of non-empty coincidence set.\par

The paper is organized as follows. In Section \ref{sec2}, we recall the definition and cohomology of Milnor manifolds. In Section \ref{sec3}, we construct free $\mathbb{Z}_n$ and $\mathbb{S}^1$-actions on these manifolds. Section \ref{sec4} consists of some preliminaries from the theory of compact transformation groups that will be used in subsequent sections. Induced action on cohomology is investigated in Section \ref{sec5}. In Section \ref{sec6}, we prove our main results as Theorems \ref{cohomo-real-z2}, \ref{cohomo-complex-z2} and \ref{cohomo-s1-real}. Finally, in Section \ref{sec7}, we deduce some Borsuk-Ulam type results for equivariant maps between spheres and these spaces. For the complex case, we obtain a lower bound on the Schwarz genus, which establishes the existence of coincidence points for maps to the Euclidean plane. \par
\bigskip

\section{Milnor manifolds}\label{sec2}
Let $r$ and $s$ be integers such that $0\leq s\leq r$. A real Milnor manifold, denoted by $\mathbb{R}H_{r,s}$, is the non-singular hypersurface of degree $(1,1)$ in the product $\mathbb{R}P^r \times \mathbb{R}P^s$ of real projective spaces. Milnor \cite{milnor3} introduced these manifolds in search for generators for the unoriented cobordism algebra. Clearly, $\mathbb{R}H_{r,s}$ is a $(s+r-1)$-dimensional closed smooth manifold, and can be described in terms of homogeneous coordinates of real projective spaces as $$\mathbb{R}H_{r,s}=\Big\{\big([x_0, \dots, x_r], [y_0,\dots,y_s]\big) \in \mathbb{R}P^r \times \mathbb{R}P^s~|~ x_0y_0 +\hspace{1mm}\cdots \hspace{1mm}+ x_sy_s=0\Big\}.$$ Alternatively, $\mathbb{R}H_{r,s}$ is given as the total space of the fiber bundle $$\mathbb{R}P^{r-1} \stackrel{i}{\hookrightarrow} \mathbb{R}H_{r,s} \stackrel{p}{\longrightarrow} \mathbb{R}P^{s}.$$ This is projectivization of the vector bundle $$\mathbb{R}^r \hookrightarrow E^{\perp} \longrightarrow \mathbb{R}P^{s},$$ where $ E^{\perp} $ is the orthogonal complement in $\mathbb{R}P^s \times \mathbb{R}^{r+1}$ of the canonical line bundle $$\mathbb{R} \hookrightarrow E \longrightarrow \mathbb{R}P^{s}.$$ \par

Similarly, a complex Milnor manifold, denoted by $\mathbb{C}H_{r,s}$, is a $2(s+r-1)$-dimensional closed
smooth manifold, given in terms of homogeneous coordinates as
$$\mathbb{C}H_{r,s}=\Big\{\big([z_0,\dots,z_r],[w_0,\dots,w_s]\big)\in \mathbb{C}P^r \times \mathbb{C}P^s~|~ z_0\overline{w}_0+\hspace{1mm}\cdots\hspace{1mm}+z_s\overline{w}_s=0\Big\}.$$ As in the real case, $\mathbb{C}H_{r,s}$ is the total space of the fiber bundle $$\mathbb{C}P^{r-1} \stackrel{i}{\hookrightarrow} \mathbb{C}H_{r,s} \stackrel{p}{\longrightarrow} \mathbb{C}P^{s}.$$ It is known due to Conner and Floyd \cite[p.63]{conner2} that $\mathbb{C}H_{r,s}$ is unoriented cobordant to $\mathbb{R}H_{r,s} \times \mathbb{R}H_{r,s}$.

These manifolds have been well-studied in the past. See, \cite{galvez, kamata, msinghmilnor} for some recent results. Their cohomology algebra is also well-known \cite{bukhshtaber, hkmukherjee}, and we need it for the proofs of our main results.

\begin{theorem}\label{lem2.1}
Let $0 \leq s \leq r$. Then the following holds:
\begin{enumerate}
\item $H^*(\mathbb{R}H_{r,s}; \mathbb{Z}_2) \cong \mathbb{Z}_2[a,b]/\langle a^{s+1}, b^r+ab^{r-1}+ \cdots+a^sb^{r-s} \rangle$,\\
where $a$ and $b$ are homogeneous elements of degree one each.\vspace{2mm}
\item $H^*(\mathbb{C}H_{r,s}; \mathbb{Z}_2) \cong \mathbb{Z}_2[g,h]/\langle g^{s+1},h^r+gh^{r-1}+\cdots+g^sh^{r-s} \rangle$,\\
where $g$ and $h$ are homogeneous elements of degree two each.
\end{enumerate}
\end{theorem}

Note that, $\mathbb{R}H_{r,0}= \mathbb{R}P^{r-1}$ and $\mathbb{C}H_{r,0}= \mathbb{C}P^{r-1}$. Since orbit spaces of free involutions on real and complex projective spaces are well-known, we, henceforth, assume that $1 \leq s \leq r$.\par
\bigskip

\section{Free actions on Milnor manifolds}\label{sec3}

\subsection{Circle Actions}
We give examples of free $\mathbb{S}^1$-actions on $\mathbb{R} H_{r,s}$ in the case when both $r$ and $s$ are odd, and later prove that this is indeed a necessary condition for the existence of a free $\mathbb{S}^1$-action. We first give a free $\mathbb{S}^1$-action on $\mathbb{R}P^{s}$. Note that, only odd-dimensional real projective spaces admit free $\mathbb{S}^1$-actions. Let $s=2m+1$ and write an element of \hspace{.5mm}$\mathbb{R}P^{2m+1}$ as $ [w_0, \dots,w_m]$, where $w_i$ are complex numbers.\par Define a map $\mathbb{S}^1 \times \mathbb{R}P^{2m+1} \rightarrow \mathbb{R}P^{2m+1}$  by 
$$\big(\xi, [w_0,\dots,w_m]\big)= [\sqrt \xi w_0,\dots,\sqrt \xi w_m].$$\label{*}
It can be checked that the preceding map gives a free $\mathbb{S}^1$-action on $\mathbb{R}P^{2m+1}$.\par

Let $r=2n+1$ and write an element of $\mathbb{R}P^{r}$ as $ [z_0 , \dots,z_n]$, where $z_i$ are complex numbers. Define an action of \hspace{.5mm}$\mathbb{S}^1$ on $\mathbb{R}P^{r}$ by 
$$\big(\xi,[z_0 ,\dots,z_n]\big)= [\sqrt \xi z_0, \dots,\sqrt \xi z_n].$$
The diagonal action on $\mathbb{R }P^r \times \mathbb{R}P^s$ is free, and $\mathbb{R} H_{r,s}$ is invariant under this action giving rise to a free $\mathbb{S}^1$-action on $\mathbb{R} H_{r,s}$. Restricting the above $\mathbb{S}^1$-action gives free $\mathbb{Z}_n$ (in particular $\mathbb{Z}_2$) action on  $\mathbb{R} H_{r,s}$.\par
\medskip
If $X$ and $Y$ are two spaces, then $X \cong Y$ means that $X$ and $Y$ have isomorphic mod 2 cohomology algebras, not necessarily induced by a map between $X$ and $Y$. For the complex case, we have the following result.

\begin{proposition}
There is no free $\mathbb{S}^1$-action on a compact Hausdorff space $X \simeq_2  \mathbb{C}H_{r,s} $.
\end{proposition}

\begin{proof}
 Recall that, we have a fiber bundle 
$$\mathbb{C}P^{r-1} \stackrel{}{\hookrightarrow} \mathbb{C}H_{r,s} \stackrel{}{\longrightarrow} \mathbb{C}P^{s}$$
with  $\chi(\mathbb{C}H_{r,s} )=r(s+1)$. Suppose there is a free $\mathbb{S}^1$-action on $X$. Restriction of this action gives free $\mathbb{Z}_p$-actions for each prime $p$. By Floyd's Euler characteristic formula \cite[Chapter III, Theorem 7.10]{bredon}, we have \[r(s+1) = \chi(X)= p\, \chi(X/\mathbb{Z}_p)\] for each prime $p$, which is a contradiction. Hence, there is no free  $\mathbb{S}^1$-action on a space $X \simeq_2  \mathbb{C}H_{r,s} $.
\end{proof}
\bigskip

\subsection{Involutions}\label{sec3.2}
If $s =r$, then interchanging the coordinates i.e.,$$\big([z_0,\dots,z_s], [w_0,\dots,w_s]\big) \longmapsto \big([w_0,\dots,w_s], [z_0,\dots,z_s]\big)$$ gives a free involution on Milnor manifolds. But, if $1<s< r$ and $r\not\equiv 2 \pmod 4$, then we show that $\mathbb{R}H_{r,s}$ (respectively $\mathbb{C}H_{r,s}$) admits a free involution if and only if both $r$ and $s$ are odd. We have seen examples of free involutions on real Milnor manifolds before.\par
For the complex case, it is known that $\mathbb{C}P^n$ admits a free action by a finite group if and only if $n$ is odd, and in that case the only possible group is $\mathbb{Z}_2$ \cite{hatcher}.\par

 If $s$ is odd, then the map
$$[z_0,z_1,\dots, z_{s-1},z_s] \longmapsto [-\overline{z}_1,\overline{z}_0,\dots, -\overline{z}_s,\overline{z}_{s-1}],$$ 
defines a free involution on $\mathbb{C}P^s$. Similarly, for $r$ odd, the map $$[z_0,z_1,\dots,z_{r-1},z_r] \longmapsto [-\overline{z}_1,\overline{z}_0,\dots, -\overline{z}_{r},\overline{z}_{r-1}],$$ is a free involution on $\mathbb{C}P^r$. Hence, the diagonal action on $\mathbb{C}P^r \times \mathbb{C}P^s$ is free and its restriction gives a free involution on $\mathbb{C}H_{r,s}$.
\bigskip

\section{Preliminaries}\label{sec4}
For the convenience of the reader, we recall some facts that we use without mentioning explicitly. For further details, we refer the reader to \cite{allday, bredon, mccleary}. Throughout, we use cohomology with $\mathbb{Z}_2$ coefficients, and suppress it from the notation.\par 

Let $G$ be a group and $X$ a $G$-space. Let 
$$G \hookrightarrow E_G \longrightarrow B_G$$ 
be the universal principal $G$-bundle and $$X\stackrel{i}{\hookrightarrow} X_G \stackrel{\pi}{\longrightarrow} B_G$$ the associated Borel fibration  \cite[Chapter IV]{borel}. Our main computational tool is the Leray-Serre spectral sequence associated to the Borel fibration \cite[Theorem 5.2]{mccleary}. The $E_2$-term of this spectral sequence is given by $$ E_2^{k,l}= H^k\big(B_G; \mathcal{H}^l(X)\big),$$ where $\mathcal{H}^l(X)$ is a locally constant sheaf with stalk $H^l(X)$ and group $G$. Further, the spectral sequence converges to $H^*(X_G)$ as an algebra. If $\pi_1(B_G)$ acts trivially on $H^*(X)$, then the system of local coefficient is simple and we get $$E_2^{k,l} \cong H^k(B_G) \otimes H^l(X).$$ Further, if the system of local coefficient is simple, then the edge homomorphisms
$$ H^k(B_G)=E_2^{k,0} \longrightarrow E_3^{k,0}\longrightarrow \cdots   \longrightarrow E_k^{k,0} \longrightarrow E_{k+1}^{k,0}=E_{\infty}^{k,0}\subset H^k(X_G)$$
and  $$H^l(X_G) \longrightarrow E_{\infty}^{0,l}= E_{l+1}^{0,l} \subset E_{l}^{0,l} \subset \cdots \subset E_2^{0,l}= H^l(X)$$
are the homomorphisms $$\pi^*: H^k(B_G) \to H^k(X_G) ~ ~ ~ \textrm{and} ~ ~ ~ i^*: H^l(X_G)  \to H^l(X),$$
respectively \cite[Theorem 5.9]{mccleary}.\par

On passing to quotients, the $G$-equivariant projection $X \times E_G \to X$ yields the fiber bundle $$E_G \hookrightarrow X_G \stackrel{h}{\longrightarrow} X/G$$ with contractible fiber $E_G$. By \cite[p. 20]{allday}, $h$ is a homotopy equivalence, and consequently
$$h^*: H^*(X/G) \stackrel{\cong}{\longrightarrow} H^*(X_G).$$

Next, we recall some results regarding free $\mathbb{Z}_2$ and  $\mathbb{S}^1$-actions on compact Hausdorff spaces. For free actions, vanishing of $H^*(X)$ implies vanishing of $H^*(X/G)$ in higher range \cite[p. 374, Theorem 1.5]{bredon}.
 
\begin{proposition}\label{z2-higher-coh-vanish}
Let $G = \mathbb{Z}_2$ act freely on a compact Hausdorff space $X$. Suppose that $H^j(X) = 0$ for all $j > n$, then $H^j(X/G) =0$ for all $j > n$. 
\end{proposition}
For $G=\mathbb{S}^1$, one can derive an analogue of the preceding result by using the Gysin-sequence for the principle bundle $X \to X/G$.
\begin{proposition}
Let $G = \mathbb{S}^1$ act freely on a compact Hausdorff space $X$. Suppose that $H^j(X) = 0$ for all $j > n$, then $H^j(X/G) =0$ for all $j \ge n$. 
\end{proposition}

We use the wellknown facts that $H^*(B_{\mathbb{Z}_2}; \mathbb{Z}_2)= \mathbb{Z}_2[t]$ and  $H^*(B_{\mathbb{S}^1}; \mathbb{Z}_2)= \mathbb{Z}_2[u]$, where  $\deg(t)=1$ and $\deg(u)=2$, respectively. 
\bigskip

\section{Induced action on cohomology}\label{sec5}
When a group acts on a topological space, in general, it is difficult to determine the induced action on cohomology. In our context, we have the following

\begin{proposition}\label{indu-real-z2}
Let $G= \mathbb{Z}_2$ act freely on a compact Hausdorff space $X \simeq_2  \mathbb{R}H_{r,s} $, where $1 < s < r$ and $r\not\equiv 2 \pmod 4$. Then the induced action on $H^*(X)$ is trivial.
\end{proposition}
\begin{proof}
Let $G=\langle g \rangle$ and $a, b \in H^1(X)$ be generators of the cohomology algebra $H^*(X)$. By the naturality of cup product, we have  $$g^*(a^ib^j)=(g^*(a))^i(g^*(b))^j$$ for all $i,~ j \geq 0$. Therefore, it is enough to consider $$g^*: H^1(X) \to H^1(X).$$\par 

Suppose that $g^*$ is non-trivial. Then it cannot preserve both $a$ and $b$. Assuming that $g^*(b) \neq b$, we have $g^*(b) = a$ or $a+b$. If $g^*(b) = a$, then $g^*(b^{s+1}) = a^{s+1}=0$, which implies $b^{s+1}=0$. Hence, $b^{r}=0$, contradicting the fact that top dimensional cohomology must be non-zero with $\mathbb{Z}_2$ coefficients. So, we must have $g^*(b) = a+b$ and $g^*(a) = a$. Suppose that $r$ is odd. Then 
$$g^*(a^{s-1}b^r) = a^{s-1}(a+b)^r=ra^sb^{r-1}+a^{s-1}b^r=a^sb^{r-1}+a^{s-1}b^r=0.$$
This gives $a^{s-1}b^r=0$, which is a contradiction. Hence, for $r$ odd, the induced action on $H^*(X)$ must be trivial.\par
Suppose that $r\equiv 0 \pmod 4$. Notice that, $b^{r+1}=0$ implies $g^*(b^{r+1}) = (a+b)^{r+1}=0$. But, from the binomial expansion $$(a+b)^{r+1}= a^{r+1}+\cdots +\binom{r+1}2 a^2b^{r-1}+(r+1) ab^r,$$ we see that the last term is non-zero and the second last term is zero modulo 2. This gives $(a+b)^{r+1}\neq 0$, a contradiction. Hence, the induced action must be trivial in this case as well.
\end{proof}
\vspace{.1mm}
\begin{remark}
If $s=1$ and $r >1$ is an odd integer, then orders of $b$ and $a+b$ are $r+1$ and $r$, respectively. Hence, in this case also $g^*$ is identity.  For $s=1$ or $r\equiv 2 \pmod4$, the induced action on $H^*(X)$ may be non-trivial. 
\end{remark}

Similarly, for the complex case, we have the following

\begin{proposition}\label{indu-complex-z2}
Let $G= \mathbb{Z}_2$ act freely on a compact Hausdorff space $X \simeq_2 \mathbb{C}H_{r,s} $, where $1 < s < r$ and $r \not\equiv2 \pmod 4$. Then the induced action on $H^*(X)$ is trivial.
\end{proposition}
\vspace{.1mm}
\begin{remark}
If $r=s$, then the free involution on $\mathbb{R}H_{s,s}$ given by
$$\big([x_0,\dots,x_s], [y_0,\dots,y_s]\big) \longmapsto \big([y_0,\dots,y_s], [x_0,\dots,x_s]\big),$$
and the similar free involution on $\mathbb{C}H_{s,s}$ given by
$$\big([z_0,\dots,z_s], [w_0,\dots,w_s]\big) \longmapsto \big([w_0,\dots,w_s], [z_0,\dots,z_s]\big)$$ 
are cohomologically non-trivial \cite[Propositions 5.1 and 5.2]{msinghmilnor}.
\end{remark}
\vspace{.5mm}
Next, we determine conditions on $r$ and $s$ for which a compact Hausdorff space $X \simeq_2 \mathbb{R}H_{r,s}$ or $\mathbb{C}H_{r,s}$ admits a free involution.

\begin{proposition}\label{real-z2-s-is-odd}
Let $G= \mathbb{Z}_2$ act freely on  $X \simeq_2 \mathbb{R}H_{r,s} $, where $1 < s < r$ and $r\not\equiv 2 \pmod 4$. Then both $r$ and $s$ are odd. 
\end{proposition}

\begin{proof}
Suppose $ \mathbb{Z}_2$ acts freely on $X \simeq_2 \mathbb{R}H_{r,s}$. Let $a$, $b \in H^1(X)$ be generators of the cohomology algebra
$H^*(X)$. 
By Proposition \ref{indu-real-z2}, $\pi_1(B_G)=\mathbb{Z}_2$ acts trivially on $H^*(X)$, so that the fibration $X  \hookrightarrow X_G  \longrightarrow B_G $ has a simple system of local coefficients. Hence the spectral sequence has the form $$E_2^{p,q} \cong H^p(B _G ) \otimes H^ q (X).$$
If $d_2:E_2^{0,1}\to E_2^{2,0}$ is trivial, then the spectral sequence degenerates at $E_2$-term and we get $H^i(X/G)\neq 0$ for infinitely many values of $i$. This contradicts Proposition \ref{z2-higher-coh-vanish}.  
Thus $d_2$ must be non-trivial. Hence we have following three possibilities:
\begin{enumerate}
\item[(i)] $d_2(1\otimes a)= t^2 \otimes 1$ and $d_2(1\otimes b)= 0$.
\item[(ii)] $d_2(1\otimes a) =0 $ and $d_2(1\otimes b) = t^2 \otimes 1$.
\item[(iii)] $d_2(1\otimes a) =t^2 \otimes 1$ and $d_2(1\otimes b)=  t^2 \otimes 1$.
\end{enumerate}
We first prove that the cases (i) and (ii) are not possible.\par
\medskip
Assuming that $s$ is odd, first we show that case (i) is not possible. The even case follows similarly. Suppose $d_2(1\otimes a)= t^2 \otimes 1$ and $d_2(1\otimes b)= 0$. By the derivation property of the differential, we have
\begin{displaymath}
d_2(t^k \otimes a^mb^n)= \left\{   \begin{array}{ll}
t^{k+2}\otimes a^{m-1}b^{n} & \textrm{if $m$ is odd}\\
0 & \textrm{if $m$ is even.}\\
\end{array} \right.
\end{displaymath}
Then the relation  $b^r + ab^{r-1} + \cdots + a^s b^{r-s}=0$ gives
\begin{align*}
  0    &=   d_2(1\otimes (b^r + ab^{r-1} + \cdots + a^s b^{r-s}))\\
   &=  d_2(1\otimes b^r) + d_2(1\otimes ab^{r-1}) + \cdots + d_2(1\otimes a^s b^{r-s})\\
   &= 0+ t^{2}\otimes b^{r-1}+ \cdots + t^{2}\otimes a^{s-1}b^{r-s}\\
   &=t^{2}\otimes (b^{r-1}+ \cdots + a^{s-1}b^{r-s}),
\end{align*}
a contradiction. Hence case (i) is not possible. The same argument works for case (ii) as well.\par 
Hence, we must have $d_2(1 \otimes a) = t^2 \otimes 1$ and $d_2(1\otimes b)=  t^2 \otimes 1$. If $s$ is even, then $a^{s+1}=0$ gives $$0= d_2(1 \otimes a^{s+1})= t^2 \otimes a^s,$$ a contradiction. Therefore, $s$ must be odd, and a similar argument shows that $r$ is also odd.
\end{proof}

As a consequence of Proposition \ref{real-z2-s-is-odd} and the previously defined $\mathbb{S}^1$-action on $\mathbb{R} H_{r,s}$, we obtain the following

\begin{corollary}
Let $1 < s< r$ and $r\not\equiv 2 \pmod 4$. Then $\mathbb{R}H_{r,s}$ admits a free involution if and only if both $r$ and $s$ are odd.
\end{corollary}
\medskip
We have similar observations for the complex case.
\begin{proposition}\label{complex-z2-s-is-odd}
Let $ \mathbb{Z}_2$ act freely on  $X \simeq_2 \mathbb{C}H_{r,s} $ with $1 < s < r$ and $r\not\equiv 2 \pmod 4$. Then both $r$ and $s $ are odd. 
\end{proposition}

\begin{corollary}
Let $1< s< r$ and $r\not\equiv 2 \pmod 4$. Then $\mathbb{C}H_{r,s}$ admits a free involution if and only if both $r$ and $s$ are odd.
\end{corollary}

For $\mathbb{S}^1$-actions, we have the following

\begin{proposition}\label{s1-action-real}
Let $1\leq s \leq r$. Then $\mathbb{S}^1$ acts freely on  $ \mathbb{R}H_{r,s} $ if and only if both $r$ and $s $ are odd. 
\end{proposition}

\begin{proof}
For $G=\mathbb{S}^1$, since $\pi_1(B_G)=1$, the system of local coefficients is simple. Recall that, $H^*(B_{\mathbb{S}^1}; \mathbb{Z}_2)= \mathbb{Z}_2[u]$, where $\deg(u)=2$. Hence, the spectral sequence has the form $$E_2^{p,q} \cong H^p(B _G ) \otimes H^ q (X).$$ Clearly, for $p$ odd, $ E_2^{ p,q} = 0$. As in the case of $\mathbb{Z}_2$-action, it  can be seen that the differential $d_2$ must be non-zero and the only possibility for $d_2$ is $d_2(1 \otimes a) = u \otimes 1$ and $d_2(1 \otimes b) = u \otimes 1$. Consequently, both $r$ and $s $ must be odd.
\end{proof}
\bigskip

\section{Main results}\label{sec6}
We are now in a position to present our main results.\smallskip
\begin{theorem}\label{cohomo-real-z2} Let $G = \mathbb{Z}_2$ act freely on a compact Hausdorff space $X \simeq_2 \mathbb{R}H_{r,s} $ such that induced action on mod 2 cohomology is trivial.   Then  
  $$ H^*(X/G;\mathbb{Z}_2) \cong \mathbb{Z}_2[x,y,z,w]/ I, $$ 
  where  \[I=\Big\langle z^2, ~w^2-\gamma_1 zw - \gamma_2 x - \gamma_3 y, x^{\frac{s+1}{2}} + \alpha_0zwx^{\frac{s-1}{2}} + \alpha_1\hspace{.5mm}zwx^{\frac{s-3}{2}}y+\cdots+ \alpha_{\frac{s-1}{2}}\hspace{.5mm}zwy^{\frac{s-1}{2}},\]  \[  \hspace{3mm} (w+\beta_0z)y^{\frac{r-1}{2}}+\hspace{.5mm}(w+\beta_1z)xy^{\frac{r-3}{2}}+\cdots+ (w+\beta_{\frac{s-1}{2}}z)\hspace{.5mm}x^{\frac{s-1}{2}}y^{\frac{r-s}{2}}\Big\rangle ,\] 
with $\deg(x)=2$, $\deg(y)=2$, $\deg(z)=1$, $\deg(w)=1$ and $\alpha_i, \beta_i, \gamma_i \in \mathbb{Z}_2$.
  \end{theorem}
  
\begin{proof}
Let $a,b \in H^1(X)$ be generators of the cohomology algebra $H^*(X)$. By similar argument as in Proposition \ref{real-z2-s-is-odd}, we see that both $r$ and $s$ must be odd and $$d_2(1\otimes a) =t^2 \otimes 1 \hspace{2mm} \textrm{and} \hspace{2mm} d_2(1\otimes b)= t^2 \otimes 1.$$
By the derivation property of the differential, we have
\begin{displaymath}
d_2(1 \otimes a^mb^n) = \left\{ \begin{array}{ll}
t^2 \otimes a^{m-1}b^n +t^2 \otimes a^mb^{n-1} & \textrm{if $m$ and $n$ are odd}\\
t^2 \otimes a^{m-1}b^n & \textrm{if $m$ is odd and $n$ is even}\\
t^2 \otimes a^mb^{n-1} & \textrm{if $m$ is even and $n$ is odd}\\
0 & \textrm{if $m$ and $n$ are even.}\\
\end{array} \right.
\end{displaymath}
 It suffices to look at $$d_2:E_2^{0,q} \to E_2^{2,q-1}.$$
$\bullet$ For $q \leq s$, a basis of $E_2^{0,q} \cong H^q(X)$ consists of $$\{ a^q, a^{q-1}b, \dots,\hspace{1mm}ab^{q-1}, b^q \}.$$
If $q$ is even, then $\rk(\Ker d_2)= \frac{q}{2}+1$ and $\rk(\im d_2)= \frac{q}{2}$. If $q$ is odd, then $\rk(\Ker d_2)= \frac{q+1}{2}= \rk(\im d_2)$.\\
$\bullet$ For $s < q \leq r-1$, a basis consists of $$\{ a^sb^{q-s}, a^{s-1}b^{q-s+1}, \dots,\hspace{1mm}ab^{q-1}, b^q \}.$$ 
In this case, $\rk(\Ker d_2)= \frac{s+1}{2}= \rk( \im d_2)$.\\
$\bullet$ For $r\leq q \leq s+r-1$, a basis consists of $$\{ a^sb^{q-s}, a^{s-1}b^{q-s+1}, \dots, \hspace{1mm}a^{q-r+1}b^{r-1} \}.$$
If $q$ is odd, then $\rk(\Ker d_2)= \frac{s+r-1-q}{2}$ and $\rk(\im d_2)= \frac{s+r+1-q}{2}$. And, if $q$ is even, then $\rk(\Ker d_2)= \frac{r+s-q}{2} =  \rk(\im d_2)$.\par
 From the above observation, we get that for all $k \geq 2$ and for all $l$, $E_3^{k,l}=0$ as $\rk (E_3^{k,l})$ = 0. This gives
 \begin{displaymath}
 E_3^{k,l}= \left\{   \begin{array}{ll}
\Ker \{ d_2: E_2^{k,l} \to E_2^{k+2, l-1} \} & \textrm{$k$ = 0, 1 and for all $l$.}\\
   0 & \textrm{$k \geq 2$ and for all $l$.}\\
 \end{array} \right.
\end{displaymath}
 
 Note that $d_r : E_r^{k,l} \rightarrow  E_r^{k+r,l-r+1}$ is trivial for all $r \geq 3 $ and for all $k,l$, and hence $ E_\infty^{*,*} \cong E_3^{*,*}.$
Since $H^*(X_G) \cong$ Tot$E_{\infty}^{*,*}$, the total complex of $E_{\infty}^{*,*}$, we have $$H^n(X_G)\cong \bigoplus_{i+j=n}E_{\infty}^{i,j}= E_{\infty}^{0,n} \oplus E_{\infty}^{1,n-1}$$ for all $0 \leq n \leq r+s-1$.\par

Note that $t \otimes 1$ is a permanent cocycle and let $z=\pi^*(t) \in E_\infty^{1,0} \subseteq H^1(X_G) $ be determined by $t \otimes 1 \in E_2^{1,0}$. As $E_\infty^{2,0} = 0$, we have $z^2 = 0$.  Also, $1 \otimes (a+b) \in E_2^{0,1}$ is a permanent cocycle. Let $w \in H^1(X_G)$ such that $i^*(w)= a+b$.  Notice that, ${ 1 \otimes a^2 } \in  E_2^{0,2}$  and $1 \otimes b^2 \in E_2^{0,2}$  are permanent   
cocycles, and hence they determine elements in $E_\infty^{0,2} $. Let $x, y \in H^2(X_G)$ such that $i^*(x)=a^2$ and $i^*(y)=b^2$. As $ a^{s+1}=0, $ we get the following relation $$x^{\frac{s+1}{2}} + \alpha_0zwx^{\frac{s-1}{2}} + \alpha_1\hspace{.5mm}zwx^{\frac{s-3}{2}}y+\cdots+ \alpha_{\frac{s-1}{2}}\hspace{.5mm}zwy^{\frac{s-1}{2}}=0,$$ where $\alpha_i\in \mathbb{Z}_2$.
%Let $y \in H^2(X_G)$ such that $i^*(y)= b^2$.
Notice that, $$i^*(wy^{\frac{r-1}{2}}+wxy^{\frac{r-3}{2}}+\cdots+ wx^{\frac{s-1}{2}}y^{\frac{r-s}{2}})=0.$$
Hence it satisfies
\[wy^{\frac{r-1}{2}}+wxy^{\frac{r-3}{2}}+\cdots+ wx^{\frac{s-1}{2}}y^{\frac{r-s}{2}}=\beta_0zy^{\frac{r-1}{2}}+\beta_1zxy^{\frac{r-3}{2}}+\beta_{\frac{s-1}{2}}zx^{\frac{s-1}{2}}y^{\frac{r-s}{2}},\]
where $\beta_i\in \mathbb{Z}_2$. Note that, we can write $ w^2 $ as the following 
$$w^2=\gamma_1 zw + \gamma_2 x + \gamma_3 y,$$ where $\gamma_i \in \mathbb{Z}_2$. Therefore 
$$H^*(X/G) \cong H^*(X_G) \cong \mathbb{Z}_2[x,y,z,w]/ I ,$$ where \[I=\Big\langle z^2, w^2-\gamma_1 zw - \gamma_2 x - \gamma_3 y, x^{\frac{s+1}{2}} + \alpha_0zwx^{\frac{s-1}{2}} + \alpha_1\hspace{.5mm}zwx^{\frac{s-3}{2}}y+\cdots+ \alpha_{\frac{s-1}{2}}\hspace{.5mm}zwy^{\frac{s-1}{2}},\]  \[  \hspace{3mm} (w+\beta_0z)y^{\frac{r-1}{2}}+\hspace{.5mm}(w+\beta_1z)xy^{\frac{r-3}{2}}+\cdots+ (w+\beta_{\frac{s-1}{2}}z)\hspace{.5mm}x^{\frac{s-1}{2}}y^{\frac{r-s}{2}}\Big\rangle ,\] 
with $\deg(x)=2$, $\deg(y)=2$, $\deg(z)=1$, $\deg(w)=1$ and $\alpha_i, \beta_i, \gamma_i \in \mathbb{Z}_2$.
\end{proof}\bigskip

For the complex case, we prove the following
\begin{theorem}\label{cohomo-complex-z2}
Let $G = \mathbb{Z}_2$ act freely on a compact Hausdorff space $X \simeq_2 \mathbb{C}H_{r,s}$, such that induced action on mod 2 cohomology is trivial. 
    Then  
    $$ H^*(X/G;\mathbb{Z}_2) \cong \mathbb{Z}_2[x,y,z,w]/ J, $$ 
    where \[J=\Big\langle z^3, ~w^2-\gamma_1 z^2w - \gamma_2 x - \gamma_3 y,\hspace{1mm} x^{\frac{s+1}{2}} + \alpha_0z^2wx^{\frac{s-1}{2}} + \alpha_1\hspace{.5mm}z^2wx^{\frac{s-3}{2}}y+\cdots+ \alpha_{\frac{s-1}{2}}\hspace{.5mm}z^2wy^{\frac{s-1}{2}},\] \[ (w+\beta_0z^2)y^{\frac{r-1}{2}}+\hspace{.5mm}(w+\beta_1z^2)xy^{\frac{r-3}{2}}+\cdots+ (w+\beta_{\frac{s-1}{2}}z^2)\hspace{.5mm}x^{\frac{s-1}{2}}y^{\frac{r-s}{2}}\Big\rangle ,\] 
    with $\deg(x)=4$, $\deg(y)=4$, $\deg(z)=1$,  $\deg(w)=2$ and $\alpha_i, \beta_i,\gamma_i \in \mathbb{Z}_2$.
  \end{theorem}
\begin{proof}
Let $G= \mathbb{Z}_2$ act freely on $X \simeq_2 \mathbb{C}H_{r,s} $.  Note that $E_2^{k,l}= 0$ for $l$ odd. This gives $$d_2:E_2^{k,l} \to E_2^{k+2,l-1}$$ is zero, and hence $E_3^{k,l}= E_2^{k,l}$ for all $k$, $l$. Let $a$, $b \in H^2(X)$ be generators of the cohomology algebra $H^*(X)$. As in the proof of the Theorem \ref{cohomo-real-z2}, the only possibility for $d_3$ is    $$d_3(1\otimes a)= t^3 \otimes 1 \hspace{2mm} \textrm{and} \hspace{2mm} d_3(1\otimes b)= t^3 \otimes 1.$$
Note that $r$ and $s$ must be odd. For various values of $l$, we consider the differentials \[d_3:E_3^{0,2l} \to E_3^{3,2l-2}.\] If we compute the ranks of $ \Ker d_3$ and $\im  d_3$, we get that $\rk (E_4^{k,2l})$ = 0 for all $k \geq 3$. This implies that $E_4^{k,2l}=0$ for all $k \geq 3$ and  $E_4^{k,2l}= \Ker \{ d_3: E_3^{k,2l} \to E_3^{k+3, 2l-2} \}$ for $k$ = 0, 1, 2. Also, $$d_r: E_r^{k,l} \to E_r^{k+r, l-r+1}$$
is zero for all $r \geq 4$. Hence $E_{\infty}^{*,*} \cong E_4^{*,*}$. Since $H^*(X_G) \cong$ Tot$E_{\infty}^{*,*}$, we get $$H^n(X_G)\cong \bigoplus_{i+j=n}E_{\infty}^{i,j}= E_{\infty}^{0,n} \oplus E_{\infty}^{1,n-1} \oplus E_{\infty}^{2,n-2}$$ for all $0 \leq p \leq 2(s+r-1)$.\par

Note that $ t\otimes 1 $ is a permanent cocycle and let $z=\pi^*(t) \in E_\infty^{1,0} \subseteq H^1(X_G) $ be determined by $t \otimes 1 \in E_2^{1,0}$. As $E_\infty^{3,0} = 0$, we have $z^3 = 0$. Also, $1 \otimes (a+b) \in E_2^{0,2}$ is a permanent cocycle. Let $w \in H^2(X_G)$ such that $i^*(w)= a+b$. Also, ${ 1 \otimes a^2 } \in  E_2^{0,4}$  and $1 \otimes b^2 \in E_2^{0,4}$  are permanent   cocycles, and hence they determine elements in $ E_\infty^{0,4} $. Let $x, y \in H^4(X_G)$ such that $i^*(x)=a^2$ and $ i^*(y)=b^2 $. As $ a^{s+1}=0 $, we get the following relation $$x^{\frac{s+1}{2}} + \alpha_0z^2wx^{\frac{s-1}{2}} + \alpha_1\hspace{.5mm}z^2wx^{\frac{s-3}{2}}y+\cdots+ \alpha_{\frac{s-1}{2}}\hspace{.5mm}z^2wy^{\frac{s-1}{2}}=0,$$ where $\alpha_i\in \mathbb{Z}_2$. Notice that, $$i^*(wy^{\frac{r-1}{2}}+wxy^{\frac{r-3}{2}}+\cdots+ wx^{\frac{s-1}{2}}y^{\frac{r-s}{2}})=0.$$ Hence it satisfies
\[wy^{\frac{r-1}{2}}+wxy^{\frac{r-3}{2}}+\cdots+ wx^{\frac{s-1}{2}}y^{\frac{r-s}{2}}=\beta_0z^2y^{\frac{r-1}{2}}+\beta_1z^2xy^{\frac{r-3}{2}}+\beta_{\frac{s-1}{2}}z^2x^{\frac{s-1}{2}}y^{\frac{r-s}{2}},\]
where $\beta_i\in \mathbb{Z}_2$.  Note that, $ w^2 $ satisfies the following relation
 $$w^2=\gamma_1 z^2w + \gamma_2 x + \gamma_3 y,$$ where $\gamma_i \in \mathbb{Z}_2$. Therefore $$H^*(X/G) \cong H^*(X_G) \cong \mathbb{Z}_2[x,y,z,w]/ J ,$$ 
where \[J=\Big\langle z^3, ~w^2-\gamma_1 z^2w - \gamma_2 x - \gamma_3 y, x^{\frac{s+1}{2}} + \alpha_0z^2wx^{\frac{s-1}{2}} +\cdots+ \alpha_{\frac{s-1}{2}}\hspace{.5mm}z^2wy^{\frac{s-1}{2}},\] \[ (w+\beta_0z^2)y^{\frac{r-1}{2}}+\hspace{.5mm}(w+\beta_1z^2)xy^{\frac{r-3}{2}}+\cdots+ (w+\beta_{\frac{s-1}{2}}z^2)\hspace{.5mm}x^{\frac{s-1}{2}}y^{\frac{r-s}{2}}\Big\rangle ,\] 
with $\deg(x)=4$, $\deg(y)=4$, $\deg(z)=1$,  $\deg(w)=2$ and $\alpha_i, \beta_i, \gamma_i \in \mathbb{Z}_2$.
This completes the proof.
\end{proof}
\bigskip
For $\mathbb{S}^1$ actions, we obtain the following
\begin{theorem}\label{cohomo-s1-real}
Let $G=\mathbb{S}^1$ act freely on a compact Hausdorff space $X \simeq_2 \mathbb{R}H_{r,s} $.  Then
  $$ H^*(X/G;\mathbb{Z}_2) \cong \mathbb{Z}_2[x,y,w]/ \langle  x^{\frac{s+1}{2}},~wy^{\frac{r-1}{2}}+xwy^{\frac{r-3}{2}}+\cdots+ wx^{\frac{s-1}{2}}y^{\frac{r-s}{2}}, ~w^2-\alpha x - \beta y \rangle, $$ 
where $\deg(x)=2$, $\deg(y)=2$, $\deg(w)=1$ and $\alpha, \beta \in \mathbb{Z}_2$.
\end{theorem}

\begin{proof}
By Proposition \ref{s1-action-real}, the only possibility for the differential $d_2$ is that $d_2(1 \otimes a) = u \otimes 1$, $d_2(1 \otimes b) = u \otimes 1$ and both $r$ and $s$ are odd. As in the proof of Theorem \ref{cohomo-real-z2}, if we compute the ranks of  $\Ker d_2$ and $\im d_2$, we get $\rk (E_3^{k,l})$ = 0 and hence $E_3^{k,l}=0$ for all $k \geq 1$ and for all $l$. Also,  $E_3^{0,l}= \Ker \{ d_2: E_2^{0,l} \to E_2^{2, l-1} \}$ for all $l$.\par
Note that, $d_r : E_r^{k,l} \rightarrow  E_r^{k+r,l-r+1}$ is trivial for all $r \geq 3 $ and for all $k,l$.
Hence $ E_\infty^{*,*} = E_3^{*,*}.$
Since $H^*(X_G) \cong$ Tot$E_{\infty}^{*,*}$, we have $$H^n(X_G)\cong \bigoplus_{i+j=n}E_{\infty}^{i,j}= E_{\infty}^{0,n} $$ for all $0 \leq n \leq r+s-1$.\par
We see that $1 \otimes (a+b) \in E_2^{0,1}$ is a permanent cocycle. Let $w \in H^1(X_G)$ such that $i^*(w)= a+b$. Also, ${ 1 \otimes a^2 } \in  E_2^{0,2}$  and $1 \otimes b^2 \in E_2^{0,2}$  are permanent cocycles. Hence they determine elements in $ E_\infty^{0,2}$. Let $x,y  \in H^2(X_G)$ such that $i^*(x)=a^2$ and $ i^*(y)=b^2 $.  As $ a^{s+1}=0 $, we get  $x^{\frac{s+1}{2}} =0$. Note that $$i^*(wy^{\frac{r-1}{2}}+xwy^{\frac{r-3}{2}}+\cdots+ wx^{\frac{s-1}{2}}y^{\frac{r-s}{2}})=0.$$  Hence we get the following relation 
$$wy^{\frac{r-1}{2}}+xwy^{\frac{r-3}{2}}+\cdots+ wx^{\frac{s-1}{2}}y^{\frac{r-s}{2}}=0.$$
Note that, we can write $ w^2 $ as the following   $$w^2=\alpha x + \beta y,$$ for some $\alpha, \beta \in \mathbb{Z}_2$.  Therefore 
 $$H^*(X/G) \cong H^*(X_G) \cong \mathbb{Z}_2[x,y,w]/\langle  x^{\frac{s+1}{2}}, wy^{\frac{r-1}{2}}+xwy^{\frac{r-3}{2}}+\cdots+ wx^{\frac{s-1}{2}}y^{\frac{r-s}{2}}, w^2-\alpha x - \beta y \rangle,$$
 where $\deg(x)=2$, $\deg(y)=2$, $\deg(w)=1$ and $\alpha, \beta \in \mathbb{Z}_2$.  
\end{proof}

\begin{example}
Take $r=3$ and $s=1$. Recall that, $\mathbb{R}H_{3,1}$ is a 3-dimensional closed smooth manifold. A free $\mathbb{S}^1$-action on $\mathbb{R}H_{3,1}$ gives a  principal $\mathbb{S}^1$-bundle $\mathbb{R}H_{3,1} \to \mathbb{R}H_{3,1}/\mathbb{S}^1$ with compact 2-dimensional base. Now, using the Leray-Serre spectral sequence associated to the Borel fibration, one can see that $H^1\big(\mathbb{R}H_{3,1}/\mathbb{S}^1;\mathbb{Z}_2\big)\cong\mathbb{Z}_2$. Hence, the orbit space must be $\mathbb{R}P^2$ and its cohomology algebra matches with our result for $\beta=1$.
\end{example}

\begin{corollary}
Let $G=\mathbb{S}^1$ act freely  on a compact Hausdorff space $X \simeq_2 \mathbb{R}H_{r,s} $. Then the Euler class of the principal $G$-bundle 
$X \stackrel{q}{\longrightarrow} X/G$ is zero.
\end{corollary}

\begin{proof}
From Theorem \ref{cohomo-s1-real}, we get $H^ i (X/G) =\mathbb{Z}_2$ for $i = 0, 1$ and $H^ 2 (X/G) = \mathbb{Z}_2 \oplus \mathbb{Z}_ 2$.  The Gysin sequence of the $G$-bundle  $X \stackrel{q}{\longrightarrow} X/G$ is
$$ 0{\longrightarrow} H^1(X/G) \stackrel{q^*}{\longrightarrow} H^1(X) \stackrel{}{\longrightarrow} H^0(X/G) \stackrel{\smallsmile e_{_{}}}{\longrightarrow}H^2(X/G)\stackrel{q^*}{\longrightarrow} \cdots, $$
where $e\in H ^2 (X/G)$ is the Euler class. The conclusion now follows from the sequence.
\end{proof}
\bigskip

\section{Applications to equivariant maps}\label{sec7}
Let $X$ be a compact Hausdorff space with a free involution and $\mathbb{S}^n$ the unit $n$-sphere equipped with the antipodal involution. Conner and Floyd \cite{conner} asked; for which integer $n$, there exists a $\mathbb{Z}_2$-equivariant map from $\mathbb{S}^n$ to $X$, but
no such map from $\mathbb{S}^{n+1}$ to $X$.\par

For $X = \mathbb{S}^n$, by the Borsuk-Ulam theorem, the answer to the preceding question is $n$. In the same paper, Conner and Floyd defined the index of the involution on $X$ as
$$\textrm{ind}(X) = \max ~ \{~ n ~|~ \textrm{there exists a}~ \mathbb{Z}_2 \textrm{-equivariant map}~ \mathbb{S}^n \to X \}.$$

The characteristic classes with  $\mathbb{Z}_2$ coefficients can be used to derive a cohomological criteria to study the above question. Let $w \in H^1(X/G; \mathbb{Z}_2 )$ be the Stiefel-Whitney class of the principal $G$-bundle $X \to X/G$. Conner and Floyd also defined
$$\textrm{{co-ind}}_{\mathbb{Z}_2}(X) = \max ~\{~n~|~ w^n \neq 0 \}.$$
Since $\textrm{{co-ind}}_{ \mathbb{Z}_2}(\mathbb{S}^n)$ = $n$, by \cite[(4.5)]{conner}, we obtain
$$\textrm{ind}(X)\leq \textrm{{co-ind}}_{\mathbb{Z}_2}(X).$$

Using these indices, we obtain the following results.

\begin{proposition}
Let $X \simeq_2 \mathbb{R}H_{r,s} $ be a compact Hausdorff space, where $1\leq s < r$. Then there is no $\mathbb{Z}_2$-equivariant map $\mathbb{S}^k \to X$ for $k \geq 2$. 
\end{proposition}

\begin{proof}
Take a classifying map $$f : X/G \to B_G$$ for the principal $G$-bundle $X \to X/G$. Let $\eta: X/G \to X_G$ is a homotopy inverse of the homotopy equivalence $h:X_G \to X/G$. Then $\pi \eta : X/G \to B_G$ also classifies the principal $G$-bundle $X \to X/G$, and hence it is homotopic to $f$. Therefore it suffices to consider the map $$\pi^*: H^1(B_G) \to H^1(X_G).$$ The image of the Stiefel-Whitney class of the universal principal $G$-bundle $G \hookrightarrow E_G \longrightarrow B_G$ is the Stiefel-Whitney class of $X \to X/G$.
For $X \simeq_2 \mathbb{R}H_{r,s} $, using the proof of Theorem \ref{cohomo-real-z2}, we see that $x \in H^1(X/G)$ is the Stiefel-Whitney class with $x \neq 0$ and $x^2=0$. This gives $\textrm{{co-ind}}_{\mathbb{Z}_2}(X) = 1$ and $\textrm{ind}(X)\leq 1$. Hence, there is no $\mathbb{Z}_2$-equivariant map $\mathbb{S}^k \to X$ for $k \geq 2$.
\end{proof}

\begin{proposition}
Let $X \simeq_2 \mathbb{C}H_{r,s} $ be a compact Hausdorff space, where $1\leq s < r$. Then there is no $\mathbb{Z}_2$-equivariant map $\mathbb{S}^k \to X$ for $k \geq 3$. 
\end{proposition}

\begin{proof}
From the proof of Theorem \ref{cohomo-complex-z2}, $x \in H^1(X/G)$ is the Stiefel-Whitney class with $x^2 \neq 0$ and $x^3=0$. This gives $\textrm{{co-ind}}_{\mathbb{Z}_2}(X) = 2$ and $\textrm{ind}(X)\leq 2$. Hence, there is no $\mathbb{Z}_2$-equivariant map $\mathbb{S}^k \to X$ for $k \geq 3$.
\end{proof}

Given a $G$-space $X$, Volovikov \cite{volovikov} defined another numerical index $i(X)$ as the smallest $r$ such that for some $k$, the differential
$$d _r : E_r ^{k-r,r-1} \rightarrow E_r^{ k,0}$$ in the Leray-Serre spectral sequence of the fibration $X \stackrel i \hookrightarrow X_G \stackrel\pi \longrightarrow B_G $ is non-trivial.
It is clear that $i(X) = r$ if $E_2^{ k,0} = E_3 ^{k,0} = \cdots = E _r ^{k,0}$ for all $k$ and $E _r^{ k,0} \neq E _{r+1}^{k,0}$ for some $k$. If $E _2^{ *,0} = E _\infty^{*,0},$ then $i(X)= \infty $. Thus, $i(X)$ is either an integer greater than 1 or $\infty$. Using this index, Coelho, Mattos and Santos proved the following \cite[Theorem 1.1]{coelho} result.

\begin{proposition} \label{mattos-result} 
Let $G$ be a compact Lie group and $X$, $Y$ be path-connected compact Hausdorff spaces with free $G$-actions. Suppose that $i(X) \geq m +1$ for some natural m $\geq$ 1. If $H ^{k+1} (Y /G; \mathbb{Z}_2) = 0 $ for some $1 \leq k < m$ and $0 < \rk \big(H^{k+1} (B _G)\big)$, then there is no G-equivariant map $f : X \rightarrow Y$.\end{proposition}

The preceding result together yields the following

\begin{proposition}
Suppose $\mathbb{Z}_2$ acts freely on $X \simeq_2 \mathbb{C}H_{r,s} $ and a path-connected compact Hausdorff space Y such that  $H^2(Y/G)=0$. Then there is no $\mathbb{Z}_2$-equivariant map $X \to Y.$ 
\end{proposition}
\begin{proof}
Note that, we obtained $i(X)=3$ in the proof of Theorem \ref{cohomo-complex-z2}. Now the result is a consequence of Proposition \ref{mattos-result}.
\end{proof}
\bigskip

Let $G$ be a finite group considered as a 0-dimensional simplicial complex and $X$ a paracompact space with a free $G$-action. The Schwarz genus $\textrm{g}_\textrm{free}(X,G)$ of the free $G$-space $X$ is the smallest number $n$ such that there exists a $G$-equivariant map $$X \to  G\ast \cdots \ast G,$$
 the $n$-fold join of $G$ equipped with the diagonal $G$-action. Note that for $G=\mathbb{Z}_2$, the free genus is the least integer $n$ for which there exists a $\mathbb{Z}_2$-equivariant map $f:X \to S^{n-1}$. See \cite[Chapter V]{schwarz} for the original source and \cite{bartsch, volovikov} for more details
and applications. In the literature, the free genus for $G=\mathbb{Z}_2$ is known under different names, for example, $B$-index \cite{yangii}, co-index \cite{conner}, level \cite{pfister}.

\begin{proposition}\label{genus-bound}
Let $X \simeq_2 \mathbb{C}H_{r,s} $ be a compact Hausdorff space with a free $\mathbb{Z}_2$-action. Then $\textrm{g}_\textrm{free}(X,\mathbb{Z}_2)\geq3$. In particular, there does not exist any $\mathbb{Z}_2$-equivariant map  $X \to \mathbb{S}^1$.
\end{proposition}

\begin{proof}
It follows from \cite{conner} that
 $$\textrm{g}_\textrm{free}(X,\mathbb{Z}_2) \geq \textrm{{co-ind}}_{ \mathbb{Z}_2}(X)+1,$$ and hence $\textrm{g}_\textrm{free}(X,\mathbb{Z}_2)\geq3$. 
\end{proof}
\bigskip

Let $G$ be a finite group and $X$ a $G$-space. Given a continuous map $f:X \to Y$, the coincidence set $A(f,\,k)$ is defined as
$$A(f,k)= \{x\in X~|~\exists\; \textnormal{distinct}\; g_1, \dots ,g_k \in G\; \textnormal{such that}\; f(g_1x)=\cdots =f(g_kx)\}.$$ The following result of Schwarz \cite{schwarz} relates the free genus and the coincidence set.

\begin{theorem}
Let X be a paracompact connected space with a free $\mathbb{Z}_p$-action. Suppose that $\textnormal{g}_\textnormal{free}(X, \mathbb{Z}_p)>m(p-1)$. Then for any continuous map $f:X \to \mathbb{R}^m$
$$\textnormal{g}_\textnormal{free}\big(A(f,p),\mathbb{Z}_p\big) \geq  \textnormal{g}_\textnormal{free}(X,\mathbb{Z}_p)-m(p-1).$$
In particular, the set $A(f,p)$ is non-empty. 
\end{theorem}

As a consequence of the preceding theorem and Proposition \ref{genus-bound}, we obtain the following

\begin{proposition}
Let $X \simeq_2 \mathbb{C}H_{r,s} $ be a compact Hausdorff space with a free $\mathbb{Z}_2$-action. Then any continuous map $X \to \mathbb{R}^2$ has a non-empty coincidence set.
\end{proposition}
\bigskip

\ack{Dey thanks UGC-CSIR for the Senior Research Fellowship towards this work and Singh acknowledges support from SERB MATRICS Grant MTR/2017/000018.}
\bigskip

\bibliographystyle{plain}

\end{document}